\documentclass[reqno,10pt]{amsart}
\usepackage{amsmath,amssymb,amsthm,mathrsfs,color,times,textcomp,verbatim,yfonts}
\usepackage{subcaption}
\usepackage[normalem]{ulem}
\usepackage[export]{adjustbox}
\usepackage{esint}
\usepackage{xcolor}
\usepackage{array}
\usepackage[colorlinks=true]{hyperref}
\hypersetup{urlcolor=blue, citecolor=red, linkcolor=blue}

\usepackage[square,numbers]{natbib}

\numberwithin{equation}{section}
\theoremstyle{plain}
\newtheorem{theorem}{Theorem}[section]
\newtheorem{proposition}[theorem]{Proposition}
\newtheorem{lemma}[theorem]{Lemma}

\newtheorem{corollary}[theorem]{Corollary}

\usepackage{graphicx}

\theoremstyle{definition}
\newtheorem{definition}[theorem]{Definition}

\newcommand{\dv}{\operatorname{div}}
\title{Self-similar solutions of semilinear heat equations with positive speed}
\author{Kyeongsu Choi}
\address{School of Mathematics, Korea Institute for Advanced Study, 85 Hoegiro, Dongdaemun-gu, Seoul 02455, Republic of Korea.}
\email{choiks@kias.re.kr}
\author{Jiuzhou Huang}
\address{School of Mathematics, Korea Institute for Advanced Study, 85 Hoegiro, Dongdaemun-gu, Seoul 02455, Republic of Korea.}
\email{jiuzhou@kias.re.kr}

\begin{document}
\allowdisplaybreaks
\maketitle
\begin{abstract}
    We classify the smooth self-similar solutions of the semilinear heat equation $u_t=\Delta u+|u|^{p-1}u$ in $\mathbb{R}^n\times (0,T)$  satisfying an integral condition for all $p>1$ with positive speed. As a corollary, we prove that finite time blowing up solutions of this equation on a bounded convex domain with $u(\cdot,0)\geq 0$ and $u_t(\cdot,0)\geq 0$ converges to a positive constant after rescaling at the blow-up point for all $p>1$.   
\end{abstract}
\section{Introduction}
In this paper, we consider the self-similar solutions of the and the blow up behaviour of the semilinear heat equation
\begin{equation}\label{eq0}
    u_t=\Delta u+|u|^{p-1}u=:\tilde{F}(u)\quad \text{ in }\Omega\times(0,T),
\end{equation}
where $\Omega\subset\mathbb{R}^n$ is a domain in $\mathbb{R}^n$, $p>1$ is a constant. 

 Suppose $u$ is a smooth solution to (\ref{eq0}) on $\Omega\times(0,T)$. $u$ is said to be self-similar about $(a,T)\in\Omega\times\mathbb{R}_+$. If $u(x,t)=\lambda^{\frac{2}{p-2}}u(a+\lambda(x-a),T+\lambda^2(t-T))$ for any $\lambda>0$. A fundamental tool to study self-similar solutions is the similarity variables. Define
\begin{equation}\label{wat}
    \begin{aligned}
    &y=\tfrac{x-a}{\sqrt{T-t}},\quad s=-\log(T-t),\\
    &w_{(a,T)}(y,s)=e^{-\frac{s}{p-1}}u(a+ye^{-\frac{s}{2}},T-e^{-s}), \quad (y,s)\in D_{a,T,\Omega},
\end{aligned}
\end{equation}
where 
\begin{equation}\label{datomega}
    D_{a,T,\Omega}=\{(y,s)\in\mathbb{R}^{n+1}|a+ye^{-\frac{s}{2}}\in\Omega,s> -\log T\}.
\end{equation}
Then $w:=w_{(a,T)}$ satisfies 
\begin{equation}\label{eqw}
    w_s=\Delta w-\tfrac{1}{2}y\cdot\nabla w-\tfrac{1}{p-1}w+|w|^{p-1}w=:F(w),\quad \text{ in }D_{a,T,\Omega}
\end{equation}
In particular, $u$ is self-similar about $(a,T)$ if and only if $w$ is independent of $s$, i.e. $w$ satisfies
\begin{equation}\label{eq1}
    \Delta w-\tfrac{1}{2}y\cdot\nabla w-\tfrac{1}{p-1}w+|w|^{p-1}w=0,
\end{equation}
with $(y,s)\in D_{a,T,\Omega}$. 

The first goal of this paper is to classify the self-similar solutions of (\ref{eq0}) with positive speed (i.e. $u_t(x,0)> 0$ for $x\in\mathbb R^n$) when $\Omega=\mathbb{R}^n$. 
\begin{theorem}\label{main 0}
    Suppose that $u$ is a smooth self-similar solution of (\ref{eq0}) on $\mathbb{R}^n\times(0,T)$ about $(a,T)$ with $p>1$, satisfying $u_t(x,0)>0$ for $x\in\mathbb R^n$, and one of the following conditions

    (1) $\int_{\mathbb{R}^n}|u(x,t)|^{2p}(T-t)^{\frac{2p}{p-1}}e^{-\frac{|x-a|^2}{4(T-t)}}dx<\infty, \quad\forall t\in (0,T)$;

    (2) $p>1+\sqrt{\tfrac{4}{3}}$;
    
holds. Then $u(x,t)=\kappa(T-t)^{-\frac{1}{p-1}}$, where $\kappa:=(\tfrac{1}{p-1})^{\frac{1}{p-1}}$.   
\end{theorem}
It's easy to see that (see Section \ref{pre} for details), the positivity of initial speed is equivalent to
\begin{equation}\label{eq posi}
    \tfrac{1}{p-1}w+\tfrac{1}{2}y_iw_i>0
\end{equation}
Thus, it suffices to prove
\begin{theorem}\label{main theo}
Suppose $w$ is a smooth solution of (\ref{eq1}) on $\mathbb{R}^n$ satisfying \eqref{eq posi}, and one of the following condition is satisfied

    (1) $\int_{\mathbb{R}^n}|w|^{2p}e^{-\frac{|y|^2}{4}}dy<\infty$; 
    
    (2) $p>1+\sqrt{\tfrac{4}{3}}$. 
    
Then $w$ is a constant, i.e. $w\equiv\kappa:=(\tfrac{1}{p-1})^{\frac{1}{p-1}}$. 
\end{theorem}

The study of equation (\ref{eq1}) plays an important role in the blowup analysis of solutions of (\ref{eq0}) and has attracted much attention in the past. 
Before recalling the known results, we first introduce several critical exponents:
$$
\text{ (Sobolev exponent) }\quad p_S:=\left\{
\begin{aligned}
   &+\infty, \quad n=1,2;\\
   &\tfrac{n+2}{n-2},\quad n\geq 3.
\end{aligned}
\right.
$$
$$
\text{ (Joseph-Lundgren exponent) }\quad p_{JL}:=\left\{
\begin{aligned}
   &+\infty, \quad n\leq 10;\\
   &1+4\tfrac{n-4+2\sqrt{n-1}}{(n-2)(n-10)},\quad n\geq 11.
\end{aligned}
\right.
$$
$$
\text{ (Lepin exponent) }\quad p_L:=\left\{
\begin{aligned}
   &+\infty, \quad n\leq 10;\\
   &1+\tfrac{6}{n-10},\quad n\geq 11.
\end{aligned}
\right.
$$
For $n=1,2$, $p>1$ or $n\geq 3$, $p\leq p_S$, Giga-Kohn showed that the only bounded solution of (\ref{eq1}) is $w=0,\pm\kappa$ in their landmark paper \cite{gk85 asy sim}. For $p>p_S$, most known results are about positive radial solutions. i.e. solutions of 
\begin{equation}\label{eq rad}
\begin{aligned}
    w_{rr}+(\tfrac{n-1}{r}-\tfrac{r}{2})w_r-\tfrac{w}{p-1}+w^p=0,\quad r>0;\\
    w_r(0)=0,\quad w>0.
\end{aligned}
\end{equation}
For $p<p_L$, it's proved in a series of authors in \cite{, qui20, bq,fp,gv,l,ns,t} that (\ref{eq rad}) has countable solutions for $p_S<p<p_{JL}$, at most countable solutions for $p=p_{JL}$, and finite for $p_{JL}<p<p_L$. For the case $p>p_L$, Mizoguchi \cite{mi} proved that (\ref{eq rad}) only has constant solution $\kappa$, the same result was claimed by Mizoguchi in \cite{mi2} for $p=p_L$, but the proof there seems not complete, see Pol\'a\v cik-Quittner \cite{qui20}.

As seen above, most of the previous classification of the self-similar solutions need either $w$ to be radial symmetric or the exponent $p$ to be subcritical. Our Theorem \ref{main 0} replaces these conditions with the positivity condition \eqref{eq posi} together with a mild integral condition 
$\int_{\mathbb{R}^n}|w|^{2p}e^{-\frac{|y|^2}{4}}dy<\infty$ (and this additional condition can be removed in the case $p>1+\sqrt{\tfrac{4}{3}}$). The idea originates from Colding-Minicozzi's classification of linearly stable self-shrinkers of mean curvature flow with polynomial volume growth in  \cite{cm12}. Due to the similarity of the mean curvature flow equation and equation (\ref{eq0}), this is reasonable. In fact, this relation was also exploited by Wang-Wei-Wu \cite{www} to study the $F-$stability of \eqref{eq0}, where they showed that the only bounded solutions of \eqref{eq1} satisfying \eqref{eq posi} is $w\equiv\kappa$ for $n\geq 3,p\geq P_S$ (see Proposition 5.1 of \cite{www}). Moreover, Wang-Wang-Wei \cite{www25} proved a parabolic Liouville theorem for ancient solutions of \eqref{eq0} in the supercritical case.

As indicated above, the classification of self-similar solution plays an important role in the study of blow up behaviour of (\ref{eq0}). A smooth solution $u$ of (\ref{eq0}) is said to blow up at time $T$ if 
\begin{equation*}
    \lim_{t\to T}\|u(\cdot,t)\|_{L^{\infty}(\Omega)}=\infty.
\end{equation*}
It's convenient to divide the blow-up into two types: type-I blow-up and type-II blow-up. The blow-up is said to be type-I if 
\begin{equation*}
    \limsup_{t\to T}\|u(\cdot,t)\|_{L^{\infty}}(T-t)^{\frac{1}{p-1}}<\infty.
\end{equation*}
Otherwise, it is said to be type-II. Suppose that $u$ has type-I blow up at $T$, and $(a,T)\in\Omega\times\mathbb{R}_+$ is a blow up point of $u$ (i.e. there exists a sequence $(x_i,t_i)\in\Omega\times(0,T)$, such that $(x_i,t_i)\to(a,T)$, and $|u(x_i,t_i)|\to+\infty$ as $i\to\infty$). 
Giga-Kohn and Giga-Matsui-Sasayama \cite{gk85 asy sim, gk85 sim var, gss04 R, gss04 c} proved that when $1<p<p_S$ and $\Omega=\mathbb{R}^n$ or $\Omega$ is convex domain, then only type-I blow up happens and the blow-up are all asymptotically self-similar in this case. For $p\geq p_S$, type-II blow-up do exists under certain conditions. Good surveys in this direction are Quittner-Souplet, \cite{QS book}, and Wang-Zhang-Zhang \cite{WZZ}. On the other hand. Friedman-McLeod \cite{fm} proved that when $\Omega$ is a bounded convex domain, $u_0\geq 0$, $u=0$ on $\partial\Omega$, and $u_t\geq 0$, then only type-I blow-up appears. Here, we prove the following theorem.
\begin{theorem}\label{cor}
   Suppose $\Omega\subset\mathbb{R}^n$ is a bounded convex domain with smooth boundary with $p>1$. Given a smooth function $\varphi(x)\geq 0$ in $\Omega$, suppose the initial boundary value problem
    \begin{equation}\label{ibv}
        \left\{
\begin{aligned}
   &u_t=\Delta u+|u|^{p-1}u\quad \text{ in }\Omega\times (0,T),\\
   &u(x,0)=\varphi(x),\\
   &u(x,t)=0,\quad x\in\partial\Omega,0<t<T,
\end{aligned}
\right.
\end{equation}
has a smooth solution $u$ on $\Omega\times(0,T)$ which blows up in finite time at $(a,T)\in \Omega\times\mathbb{R}_+$, 
and $u_t(x,0)\geq 0$ for $x\in\Omega$. Then $u$ is asymptotically self-similar to $\frac{\kappa}{(T-t)^{\frac{1}{p-1}}}$ as $t\to T$. Equivalently, $w(y,s)=w_{a,T}(y,s)$ converges to $\kappa$ in $C^{\infty}_{loc}(\mathbb{R}^n)$ as $s\to\infty$, where $w_{a,T}$ is defined in \eqref{wat}.   
\end{theorem}
Note that Bebernes-Eberly \cite{be} got similar results when $n\geq 3$ and $p\geq\tfrac{n}{n-2}$ for the radially-symmetric case. That is, $\Omega=\{x\in\mathbb{R}^n||x|<R\}$ being a ball centered at the origin, $u_t\geq 0$, $\varphi\geq 0$ are radial symmetric (see also Galaktionov-Posashkov \cite{gp}, Giga-Kohn \cite{gk85 asy sim} for $n=1,2$). We removed the assumption of the radial symmetry of $\Omega$ and $u$.

The rest of the paper is organized as follows. In Section \ref{pre}, we define the self-similar solutions with positive speed and  derive some basic facts about the linearized operator of $F$ defiend in \eqref{eqw}. In Section \ref{int es}, we first derive some formulas for integration by parts in a weighted space on a noncompact domain. Then we derive the integral estimates which will conclude the proof of Theorem \ref{main theo} and Theorem \ref{cor} in Section \ref{the}. In the last section, we define the linear stability of self-similar solutions, and discuss its relation with positive speed.

{\bf Acknowledgement}  
K. Choi is supported by KIAS Individual Grant MG078902, J. Huang is supported by KIAS Individual Grant MG088501.

\section{Self-similar solutions of positive speed}\label{pre}

In this section, we assume that $u$ is a smooth self-similar solution of (\ref{eq0}) w.r.t. $(a,T)$ on $\mathbb{R}^n\times(0,T)$. 
Using notations in (\ref{wat}), $u$ is self-similar w.r.t. $(a,T)$ if and only if $w=w_{a,T}(y,s)$ is independent of $s$, i.e.
\begin{equation}
    0=w_s=\Delta w-\tfrac{1}{2}y\cdot\nabla w-\tfrac{1}{p-1}w+|w|^{p-1}w=:F(w).
\end{equation}
We first compute the eigenfunctions and eigenvalues of the linearization of $F$. The linearization $L_w$ of $F$ at $w$ is:
\begin{equation}\label{op Lw}
\begin{aligned}
   L_wv=\Delta v-\tfrac{1}{2}y\cdot\nabla v-\tfrac{v}{p-1}+p|w|^{p-1}v.
    \end{aligned}
\end{equation}
We write $L_w$ as $L$ for short in the rest of the paper if there is no confusion. A non-zero $C^2$ function $v$ is called an eigenfunction of $L$ with eigenvalue $\lambda$ if it satisfies
\begin{equation}
    Lv=\Delta v-\tfrac{1}{2}y\cdot\nabla v-\tfrac{v}{p-1}+p|w|^{p-1}v=-\lambda v.
\end{equation}

We compute the eigenfunctions of $L$ corresponding to re-centering of space and time for self-similar solutions, which will be used in the discussion of linear stability of $u$ in later sections. Note that $w_s=0$ for self-similar solutions. A direct computation from \eqref{wat} gives
\begin{equation}\label{ui ut}
\begin{aligned}
u_i=&(T-t)^{-\frac{1}{p-1}-\frac{1}{2}}w_i,\quad i=1,2,\cdots n;\\
    u_t=&\,\tfrac{1}{p-1}(T-t)^{-\frac{p}{p-1}}w+(T-t)^{-\frac{1}{p-1}}(\tfrac{1}{2}w_iy_i(T-t)^{-1}+w_s\tfrac{1}{T-t})\\
    =&\,(T-t)^{-\frac{p}{p-1}}(\tfrac{1}{p-1}w+\tfrac{1}{2}y_iw_i)
\end{aligned}
\end{equation}
Ignoring the multiple constants, this suggests that $w_i$ ($i=1,2,\cdots n$) and $\tfrac{1}{p-1}w+\tfrac{1}{2}y_iw_i$ are the eigenfunctions of $L$ which correspond to the re-centering of space and time variable respectively. In fact, we have the following lemma.
\begin{lemma}\label{lemma eig t sp}
 Suppose $w$ is smooth and satisfies (\ref{eq1}) on $\mathbb{R}^n$. Then 
 \begin{equation}
     \begin{aligned}
         L(\tfrac{1}{p-1}w+\tfrac{1}{2}\sum_{i=1}^ny_iw_i)=\tfrac{1}{p-1}w+\tfrac{1}{2}y_iw_i;\\
         Lw_i=\frac{1}{2}w_i,\quad(i=1,2\cdots,n).
     \end{aligned}
 \end{equation}
 In particular, if $\tfrac{1}{p-1}w+\tfrac{1}{2}\sum_{i=1}^ny_iw_i\not\equiv 0$, it is an eigenfunction of $L$ with eigenvalue $-1$; if $w_i\not\equiv0$, it is an eigenfunction of $L$ with eigenvalue $-\frac{1}{2}$. 
\end{lemma}
\begin{proof}
    Differentiating the equation (\ref{eq1}) with respect to $y_i$ gives 
\begin{equation}
    Lw_i=\tfrac{1}{2}w_i, \quad i=1,2,\cdots n.
\end{equation}

For the function $\tfrac{1}{p-1}w+\tfrac{1}{2}y_iw_i$, a direct computation using (\ref{eq1}) shows 
 \begin{align*}
    &\Delta(w_iy_i)-\tfrac{1}{2}y\cdot\nabla (w_iy_i)=y_i\Delta w_i+2\Delta w-\tfrac{1}{2}w_{ik}y_iy_k-\tfrac{1}{2}y\cdot\nabla w,\\
    =&\,y_i(\tfrac{1}{2}y_kw_{ik}+\tfrac{1}{p-1}w_i-p|w|^{p-1}w_i+\tfrac{1}{2}w_i)-\tfrac{1}{2}w_{ik}y_iy_k\\
    &+2(\tfrac{1}{2}y\cdot \nabla w+\tfrac{1}{p-1}w-|w|^{p-1}w)-\tfrac{1}{2}y\cdot\nabla w\\
    =&\,\tfrac{y\cdot\nabla  w}{p-1}+y\cdot\nabla w-p|w|^{p-1}w_iy_i+\tfrac{2}{p-1}w-2|w|^{p-1}w.
\end{align*}
This implies that
\begin{equation}\label{wiyi}
    L(w_iy_i)=y\cdot \nabla w+\tfrac{2}{p-1}w-2|w|^{p-1}w.
\end{equation}
On the other hand, 
\begin{equation}\label{w}
    L(\tfrac{2}{p-1}w)=\tfrac{2}{p-1}L_ww=\tfrac{2}{p-1}(-|w|^{p-1}w+p|w|^{p-1}w)=2|w|^{p-1}w.
\end{equation}
Combining (\ref{wiyi}) and (\ref{w}), we get
\begin{equation}
    L(w_iy_i+2\tfrac{1}{p-1}w)=w_iy_i+\tfrac{2}{p-1}w.
\end{equation}
\end{proof}
\begin{definition}
    A smooth self-similar solution $u(x,t)$ of \eqref{eq0} with $\Omega=\mathbb R^n$ is said to have positive speed if $u_t(x,t)>0$ for $(x,t)\in\mathbb R^n\times(0,T)$.
\end{definition}
\begin{corollary}\label{positivity of H}
     Suppose $u$ is a smooth self-similar solution of (\ref{eq0})  w.r.t. $(a,T)$ on $\mathbb{R}^n\times (0,T)$ and has positive speed. Then $\tfrac{1}{p-1}w+\tfrac{1}{2}y_iw_i>0$ is a positive eigenfunction of $L$ on $\mathbb R^{n}$ with eigenvalue $-1$.
\end{corollary}
\begin{proof}
    This follows from the above lemma and note that $u_t>0$ implies $\tfrac{1}{p-1}w+\tfrac{1}{2}y_iw_i>0$ by the second equation in \eqref{ui ut}.
\end{proof}



\section{Integral estimates}\label{int es}
We assume that $u$ is a smooth self-similar solution of (\ref{eq0}) on $\mathbb{R}^n\times(0,T)$ (w. r. t. $(a,T)$) with positive speed in this section. To prove Theorem \ref{main theo} and Theorem \ref{cor} in section \ref{the}, we need some integral estimates, which will be derived in this section. The main tool is integration by parts in a weighted space on a noncompact domain. First we introduce some notations. 

We first introduce the Ornstein–Uhlenbeck operator 
\begin{equation}\label{eq op OU}
\mathcal{L}:=\Delta-\tfrac{1}{2}y\cdot\nabla.
\end{equation}
Then the linearized operator $L=L_w$ defined in \eqref{op Lw} can be written as
$$L=L_w
=\mathcal{L}-\tfrac{1}{p-1}+p|w|^{p-1}.$$
Then we introduce the weighed inner product 
\begin{equation}\label{eq we inn}
     \langle f,g\rangle_W:=\int_{\mathbb{R}^n}fge^{-\frac{|y|^2}{4}} dy,\quad f,g\in C^0(\mathbb R^n)
\end{equation}
and the notation
\begin{equation}
    [f]_W:=\int_{\mathbb{R}^n}fe^{-\frac{|y|^2}{4}}dy, \quad f\in C^0(\mathbb R^n).
\end{equation}
\begin{definition}
    A function $f\in C^{2}(\mathbb{R}^n)$ is said to in the weighted $W^{1,2}$ space if
\begin{equation}
    \int_{\mathbb{R}^n}(|f|^2+|\nabla f|^2)e^{-\frac{|y|^2}{4}}dy=[f^2+|\nabla f|^2]_W<\infty.
\end{equation}
\end{definition}

We now give some formula for integration by parts in the weighted $W^{1,2}$ space. The proof follows the corresponding results for mean curvature flow in section 3 of \cite{cm12}, we give here for completeness. First, we consider the formula for functions with compact support.
\begin{lemma}\label{int by part comp}
    If $f\in C^1(\mathbb{R}^n)$, $g\in C^2(\mathbb{R}^n)$ function, and at least one of $f,g$ has compact support. Then
    \begin{equation}
        \int_{\mathbb{R}^n}f\mathcal{L}ge^{-\frac{|y|^2}{4}}dy=-\int_{\mathbb{R}^n}\langle\nabla g,\nabla f\rangle e^{-\frac{|y|^2}{4}}dy.
    \end{equation}
    where $\mathcal L$ is the Ornstein–Uhlenbeck operator in \eqref{eq op OU}, $\langle \cdot,\cdot\rangle$ is the usual inner product on $\mathbb R^n$.
\end{lemma}
\begin{proof}
    This is just the divergence theorem since at least one of $f,g$ has compact support.
\end{proof}
For general $C^2$ functions, we have:
\begin{lemma}\label{int by part w2}
    If $f,g\in C^2(\mathbb{R}^n)$ with
    \begin{equation}\label{w2 finite}
        \int_{\mathbb{R}^n}(|f\nabla g|+|\nabla f||\nabla g|+|f\mathcal{L}g|)e^{-\frac{|y|^2}{4}}dy<\infty,
    \end{equation}
    then we get
    \begin{equation}
        \int_{\mathbb{R}^n}f\mathcal{L}ge^{-\frac{|y|^2}{4}}dy=-\int_{\mathbb{R}^n}\langle\nabla g,\nabla f\rangle e^{-\frac{|y|^2}{4}}dy.
    \end{equation}
\end{lemma}
\begin{proof}
    Given any $C^1$ function $\phi$ with compact support, we can apply Lemma \ref{int by part comp} to $\phi f$ and $g$ to get
    \begin{equation}
        [\phi f\mathcal{L}g]_W=-[\phi\langle \nabla g,\nabla f\rangle]_W-[f\langle\nabla g,\nabla\phi\rangle]_W.
    \end{equation}
    Next, we apply this with $\phi=\phi_R\geq 0$, where $\phi_R$ is a smooth cut-off function satisfying $\phi_R=1$ on the ball $B_R$ and $\phi_R=0$ on $\mathbb{R}^n\setminus B_{R+1}$ with $|\nabla\phi_R|\leq 1$. Then the dominate convergence theorem gives that
    \begin{align*}
        &[\phi_Rf\mathcal{L}g]_W\to [f\mathcal{L}g]_W,\\
       & [\phi_R\langle\nabla g,\nabla f\rangle]_W\to[\langle\nabla g,\nabla f\rangle]_W,\\
        &[f\langle\nabla g,\nabla\phi_R\rangle]_W\to0.
    \end{align*}
    due to (\ref{w2 finite}).
\end{proof}
Since $u$ is a smooth self-similar solution w.r.t. $(a,T)$ on $\mathbb R^n\times(0,T)$ with positive speed, $w$ is a smooth solution of (\ref{eq1}) in $\mathbb{R}^n$. Using the notation
\begin{equation}
    H:=\tfrac{1}{p-1}w+\tfrac{1}{2}y_iw_i
\end{equation}
to denote the eigenfunction of $L$ with eigenvalue $-1$. We have $H>0$ by lemma \ref{positivity of H}.
\begin{lemma}\label{phi f test}
Suppose that $f$ is a $C^2$ function on $\mathbb{R}^n$ with $Lf=-\mu f$ for $\mu\in\mathbb{R}$. If $f>0$ and $\phi$ is in the weighted $W^{1,2}$ space, then
\begin{equation}\label{phi f test eq}
    \int_{\mathbb{R}^n}\phi^2(p|w|^{p-1}+|\nabla \log f|^2)e^{-\frac{|y|^2}{4}}dy\leq\int_{\mathbb{R}^n}(4|\nabla\phi|^2-2(\mu-\tfrac{1}{p-1})\phi^2)e^{-\frac{|y|^2}{4}}dy.
\end{equation}
\end{lemma}
\begin{proof}
    Since $f>0$, $\log f$ is well defined and we have 
    \begin{equation}
    \begin{aligned}
        \mathcal{L}\log f=&\tfrac{\mathcal{L}f}{f}-|\nabla \log f|^2=\tfrac{Lf+(\frac{1}{p-1}-p|w|^{p-1})f}{f}-|\nabla \log f|^2\\
        =&-\mu+\tfrac{1}{p-1}-p|w|^{p-1}-|\nabla \log f|^2.
    \end{aligned}
    \end{equation}
    Suppose that $\eta$ is a function with compact support. Then, the self-adjointness of $\mathcal{L}$ (Lemma \ref{int by part comp}) gives
    \begin{equation}
        [\langle\nabla\eta^2,\nabla\log f\rangle]_W=-[\eta^2\mathcal{L}\log f]_W=[\eta^2(\mu-\tfrac{1}{p-1}+p|w|^{p-1}+|\nabla \log f|^2)]_W.
    \end{equation}
    Since 
    \begin{align*}
        \langle\nabla\eta^2,\nabla\log f\rangle=2\langle \eta \nabla\eta,\nabla\log f\rangle\leq 2|\nabla\eta|^2+\tfrac{1}{2}\eta^2|\nabla\log f|^2.
    \end{align*}
    We get
    \begin{equation}\label{wp-1 compact}
        [\eta^2(p|w|^{p-1}+|\nabla\log f|^2)]_W\leq [4|\nabla\eta|^2-2(\mu-\tfrac{1}{p-1})\eta^2]_W.
    \end{equation}
    Let $\eta_R\geq 0$ be one on $B_R$ and zero on $\mathbb{R}^n\setminus B_{R+1}$ so that $0\leq\eta\leq 1$ and $|\nabla\eta|\leq 1$. Since $\phi$ is in the weighted $W^{1,2}$ space, applying (\ref{wp-1 compact}) with $\eta=\eta_R\phi$, letting $R\to\infty$ and using the monotone convergence theorem and dominated convergence theorem gives that (\ref{wp-1 compact}) also holds with $\eta=\phi$. 
\end{proof}

\begin{proposition}
    If $H>0$, and $[|w|^{2m}]_W<\infty$ with $m^2-p(2m-1)<0$ and $m>\frac{1}{2}$. Then
    \begin{equation}
        [|w|^{2m}+|w|^{2m+p-1}+|\nabla|w|^{m}|^2]_W<\infty.
    \end{equation}
    In particular, if $p>1+\sqrt{\tfrac{4}{3}}$, we can take $m=\frac{p-1}{2}$. If $[|w|^{2p}]<\infty$, we can take $m=p$.  
\end{proposition}
\begin{proof}
    First, since $H>0$, $\log H$ is well defined and 
\begin{equation}\label{mathcal log H}
\begin{aligned}
    \mathcal{L}\log H=&-|\nabla\log H|^2+\frac{\Delta H-\frac{1}{2}y\cdot\nabla H}{H}\\
    =&-|\nabla\log H|^2+\frac{H+(\tfrac{1}{p-1}-p|w|^{p-1})H}{H}\\
    =&-|\nabla\log H|^2+\tfrac{p}{p-1}-p|w|^{p-1}.
\end{aligned}
\end{equation}
Given any compactly supported function $\phi$, self-adjointness of $\mathcal{L}$ (Lemma \ref{int by part comp}) gives
\begin{equation}
        [\langle\nabla\phi^2,\nabla\log H\rangle]_W=-[\phi^2\mathcal{L}\log H]_W=[\phi^2(-\tfrac{p}{p-1}+p|w|^{p-1}+|\nabla \log H|^2)]_W.
 \end{equation}
 Combining this with the Cauchy inequality 
 $$|\langle\nabla \phi^2,\nabla\log H\rangle|=2|\langle\phi\nabla\phi,\nabla\log H\rangle|\leq |\nabla\phi|^2+\phi^2|\nabla\log H|^2$$ 
 gives
 \begin{equation}
     [\phi^2|w|^{p-1}]_W\leq[\tfrac{1}{p-1}\phi^2+\tfrac{1}{p}|\nabla\phi|^2]_W.
 \end{equation}
 We will apply this with $\phi=\eta |w|^{m}$ where $\eta\geq 0$ is a smooth non-negative function with compact support and $m>0$ is a real number. This gives
 \begin{equation}\label{stability ineq}
 \begin{aligned}
     &[\eta^2|w|^{2m+p-1}]_W\\
     \leq&[\tfrac{1}{p}(\eta^2|\nabla|w|^{m}|^2+|\nabla\eta|^2|w|^{2m}+2\eta|w|^m\langle\nabla\eta,\nabla |w|^{m}\rangle)+\tfrac{1}{p-1}\eta^2|w|^{2m}]_W\\
     \leq&[\tfrac{1+\varepsilon}{p}\eta^2|\nabla|w|^{m}|^2]_W+[|w|^{2m}(\tfrac{1+\frac{1}{\varepsilon}}{p}|\nabla\eta|^2+\tfrac{1}{p-1}\eta^2)]_W\\
     =&\tfrac{1+\varepsilon}{p}m^2[\eta^2|w|^{2m-2}|\nabla w|^2]_W+[|w|^{2m}(\tfrac{1+\frac{1}{\varepsilon}}{p}|\nabla\eta|^2+\tfrac{1}{p-1}\eta^2)]_W,
\end{aligned}
 \end{equation}
 where $\varepsilon>0$ is arbitrary and the last inequality used the inequality $2ab\leq \varepsilon a^2+\frac{1}{\varepsilon}b^2$.

 Second, using the definition of $L$ and the fact that $w$ is a solution of (\ref{eq1}), we get that for any positive number $m'$, we have
 \begin{equation}\label{L wm}
     \begin{aligned}
         &\mathcal{L}|w|^{m'}=m'|w|^{m'-2}w\mathcal{L}w+m'(m'-1)|w|^{m'-2}|\nabla w|^2\\
         =&m'|w|^{m'-2}w(Lw+(\tfrac{1}{p-1}-p|w|^{p-1})w)+m'(m'-1)|w|^{m'-2}|\nabla w|^2\\
         =&m'|w|^{m'-2}w((p-1)|w|^{p-1}w+(\tfrac{1}{p-1}-p|w|^{p-1})w)+m'(m'-1)|w|^{m'-2}|\nabla w|^2\\
         =&m'|w|^{m'}(\tfrac{1}{p-1}-|w|^{p-1})+m'(m'-1)|w|^{m'-2}|\nabla w|^2\\
         =&m'(m'-1)|w|^{m'-2}|\nabla w|^2+\tfrac{m'}{p-1}|w|^{m'}-m'|w|^{m'+p-1}.
     \end{aligned}
 \end{equation}
 Integrating this against $\eta^2$ and using the self-adjointness of $\mathcal{L}$ (Lemma \ref{int by part comp}) gives
 \begin{equation}
 \begin{aligned}
    &-[2m'\langle\eta\nabla\eta,|w|^{m'-2}w\nabla w]_W\\
    =&[m'(m'-1)\eta^2|w|^{m'-2}|\nabla w|^2+\tfrac{m'}{p-1}\eta^2|w|^{m'}-m'\eta^2|w|^{m'+p-1}]_W.
\end{aligned}
\end{equation}
 Using the inequality $2ab\leq \varepsilon a^2+\frac{1}{\varepsilon}b^2$ again gives
 \begin{equation}\label{eign ineq}
     [\eta^2|w|^{m'+p-1}]_W+[\tfrac{1}{\varepsilon}|w|^{m'}|\nabla\eta|^2]_W\geq ((m'-1)-\varepsilon)[\eta^2|w|^{m'-2}|\nabla w|^2]_W.
 \end{equation}
 Plugging (\ref{eign ineq}) with $m'=2m$ into (\ref{stability ineq}) gives
 \begin{equation}\label{eq int}
 \begin{aligned}
   &[\eta^2|w|^{2m+p-1}]_W\\
   \leq &\tfrac{1+\varepsilon}{p}\tfrac{m^2}{2m-1-\varepsilon}[\eta^2|w|^{2m+p-1}]_W+[|w|^{2m}((\tfrac{1+\frac{1}{\varepsilon}}{p}+\tfrac{(1+\varepsilon)m^2}{p(2m-1-\varepsilon)\varepsilon})|\nabla\eta|^2+\tfrac{1}{p-1}\eta^2)]_W .
\end{aligned}
\end{equation}
In order to use the above inequality to get the upper bound for $ [\eta^2|w|^{2m+p-1}]$, we need $\tfrac{m^2}{p(2m-1)}<1$, that is,
\begin{equation}
    m^2-p(2m-1)=m^2-2pm+p=(m-p)^2-p^2+p<0,
\end{equation}
which is satisfied by the assumption on $m,p$. Thus we can take $\varepsilon>0$ sufficiently small to absorb the term $\tfrac{1+\varepsilon}{p}\tfrac{m^2}{2m-1-\varepsilon}[\eta^2|w|^{2m+p-1}]_W$ into the left hand side of \eqref{eq int} to get
\begin{equation}
    [\eta^2|w|^{2m+p-1}]_W\leq C(p,\tfrac{1}{p-1},m,\varepsilon)[|w|^{2m}(|\nabla\eta|^2+|\eta|^2)]_W.
\end{equation}
We take $\eta=\eta_R\geq 0$ such that $\eta_R=1$ on $B_R$ and $\eta_R=0$ on $\mathbb{R}^n\setminus B_{R+1}$ so that $|\nabla\eta_R|\leq 1$. Since $[|w|^{2m}]_W<\infty$, the monotone convergence theorem then implies $[|w|^{2m+p-1}]_W<\infty$ by letting $R\to\infty$. Using (\ref{eign ineq}), we get $[|\nabla |w|^m|^2]_W=[m^2|w|^{2m-2}|\nabla w|^2]_W<\infty$ by monotone convergence theorem.

If $p>1+\sqrt{\tfrac{4}{3}}$, we can take $m=\tfrac{p-1}{2}$. In fact, if we take $f=H$ and $\phi\equiv1$ in Lemma \ref{phi f test}, then (\ref{phi f test eq}) implies that $[|w|^{p-1}]_W<\infty$. On the other hand, $0<\tfrac{(p-1)^2}{4p(p-2)}<1,\tfrac{p-1}{2}>\frac{1}{2}\Leftrightarrow3p^2-6p-1>0,p>2\Leftrightarrow p>1+\sqrt{\tfrac{4}{3}}$. Thus we can take $m=\tfrac{p-1}{2}$ when $p>1+\sqrt{\tfrac{4}{3}}$.

If $[|w|^{2p}]_W<\infty$, we can take $m=p>1$, so that $\tfrac{m^2}{p(2m-1)}=\tfrac{p^2}{p(2p-1)}<1$ since $p>1$. 

\end{proof}

\begin{proposition}
If $H>0$, and $|w|^m$ is in the weighted $W^{1,2}$ space (i.e. $[|w|^{2m}+|\nabla|w|^m|^2]_W<\infty$) and $[|w|^{2m+p-1}]_W<\infty$ with $m^2-p(2m-1)\leq (<)0$ and $m>\frac{1}{2}$, then $|w|^m\nabla\log H=\nabla |w|^{m}$ (and $|w|^{2m-2}|\nabla w|^2=0$). Consequently, $\nabla\log H=\nabla \log |w|^m$ (and $\nabla w=0$) or $w=0$.    
\end{proposition}

 \begin{proof}
 Since $|w|^m$ is in the weighted $W^{1,2}$ space,
 \begin{equation}\label{eq int1}
 [|w|^{2m}|\nabla\log H|^2]_W<\infty
 \end{equation}
 by taking $\phi=|w|^m$ and $f=H$ in Lemma \ref{phi f test}. Moreover, by Cauchy inequality, we get
 \begin{align*}
 |w|^{2m}|\nabla\log H|\leq& \frac{1}{2}(|w|^{2m}+|w|^{2m}|\nabla\log H|^2)
 \end{align*}
 and\begin{align*}
     &|\nabla|w|^{2m}||\nabla\log H|=2m|w|^{2m-1}|\nabla w||\nabla\log H|\\
     \leq& m^2|w|^{2m-2}|\nabla w|^2+ |w|^{2m}|\nabla\log H|^2
     =m^2|\nabla |w|^m|^{2}+|w|^{2m}|\nabla\log H|^2.
     \end{align*}
These two inequalities and \eqref{eq int1} implies
    \begin{equation}\label{w2m gra2m}
        [|w|^{2m}|\nabla\log H|+|\nabla|w|^{2m}||\nabla\log H|]_W<\infty.
    \end{equation}
since $|w|^m$ is in the weighted $W^{1,2}$ space by assumption. Further more, since $[|w|^{2m+p-1}]<\infty$ by assumption, and
\begin{align*}
     \mathcal{L}\log H=&-|\nabla\log H|^2+\tfrac{p}{p-1}-p|w|^{p-1}
\end{align*}
by (\ref{mathcal log H}), we get
\begin{align*}
    |w|^{2m}|\mathcal{L}\log H|\leq& |w|^{2m}||\nabla\log H|^2+\tfrac{p}{p-1}+p|w|^{p-1}|\\
    \leq&  |w|^{2m}|\nabla\log H|^2+\tfrac{p}{p-1}|w|^{2m}+p|w|^{2m+p-1},
 \end{align*}
This implies
 \begin{equation}\label{gra log 2m}
     [||w|^{2m}\mathcal{L}\log H|]_W<\infty.
 \end{equation}
Combining (\ref{w2m gra2m}) and (\ref{gra log 2m}), we can apply Lemma \ref{int by part w2} (take $f=|w|^m$ and $g=\log H$ there) to get
 \begin{equation}\label{inner prod}
     \begin{aligned}
        &[\langle \nabla |w|^{2m},\nabla\log H^{\frac{m}{p}}\rangle]_W
        =-\tfrac{m}{p}[|w|^{2m}\mathcal{L}\log H]_W\\
        =&-\tfrac{m}{p}[|w|^{2m}((\tfrac{p}{p-1}-p|w|^{p-1})-|\nabla\log H|^2)]_W\\
        =&\,\tfrac{m}{p}[p|w|^{2m+p-1}-\tfrac{p}{p-1}|w|^{2m}+|w|^{2m}|\nabla\log H|^2]_W.\\
 \end{aligned} 
 \end{equation}
On the other hand,
 \begin{equation}
      \mathcal{L}|w|^{m}=m|w|^{m}(\tfrac{1}{p-1}-|w|^{p-1})+m(m-1)|w|^{m-2}|\nabla w|^2,
 \end{equation}
by (\ref{L wm}). Since $|w|^m$ is in the weighted $W^{1,2}$ space and $[|w|^{2m+p-1}]<\infty$, this together with the inequality 
 \begin{align*}
     |w|^m|\mathcal{L}|w|^m|=&|m|w|^{2m}(\tfrac{1}{p-1}-|w|^{p-1})+m(m-1)|w|^{2m-2}|\nabla w|^2|\\
     \leq& \tfrac{m}{p-1}|w|^{2m}+m|w|^{2m+p-1}+m(m-1)|\nabla|w|^m|^2
 \end{align*}
 implies
 \begin{equation}
     [|w|^m|\nabla|w|^m|+|\nabla|w|^m|^2+||w|^m\mathcal{L}|w|^m|]_W<\infty.
 \end{equation}
 Thus, we can apply Lemma \ref{int by part w2} with $f=g=|w|^m$ to get
 \begin{equation}\label{grad eq}
    \begin{aligned}
      &[|\nabla|w|^m|^2]_W=-[|w|^m\mathcal{L}|w|^m]_W\\
      =&m[|w|^{2m}(|w|^{p-1}-\tfrac{1}{p-1})]_W-m(m-1)[|w|^{2m-2}|\nabla w|^2]_W.
 \end{aligned}  
 \end{equation}

Combining (\ref{inner prod}) and (\ref{grad eq}) gives
\begin{equation}\label{eq inne}
    \begin{aligned}
    &[\langle \nabla |w|^{2m},\nabla\log H^{\frac{m}{p}}\rangle]_W\\
    =&[|\nabla|w|^m|^2+m(m-1)|w|^{2m-2}|\nabla w|^2+\tfrac{m}{p}|w|^{2m}|\nabla\log H|^2]_W\\
    =&[m(2m-1)|w|^{2m-2}|\nabla w|^2+\tfrac{m}{p}|w|^{2m}|\nabla\log H|^2]_W.
\end{aligned}
\end{equation}
However,
\begin{equation}\label{ineq inne}
\begin{aligned}
    &[\langle \nabla |w|^{2m},\nabla\log H^{\frac{m}{p}}\rangle]_W
    =[2m\tfrac{m}{p}\langle|w|^{2m-2}w\nabla w,\nabla\log H\rangle]_W\\
    \leq& [m^2\tfrac{m}{p}|w|^{2m-2}|\nabla w|^2+\tfrac{m}{p}|w|^{2m}|\nabla\log H|^2]_W.\\
    \end{aligned}
\end{equation}
Thus, (\ref{eq inne}) and (\ref{ineq inne}) implies
\begin{equation}\label{eq fin}
   [\tfrac{m}{p}||w|^m\nabla\log H-m|w|^{m-2}w\nabla w|^2+m(2m-1-\tfrac{m^2}{p})|w|^{2m-2}|\nabla w|^2]_W\leq 0.
\end{equation}
Since $m>\frac{1}{2}$ and $(2m-1)\geq(>)\tfrac{m^2}{p}$ hold if $m^2-p(2m-1)\leq (<)0$. This implies "=" holds in (\ref{eq fin}), and we have $|w|^m\nabla\log H=m|w|^{m-2}w\nabla w=\nabla |w|^{m}$ (and $|w|^{2m-2}|\nabla w|^2=0$) $\Leftrightarrow\nabla\log H=\nabla \log |w|^m$ (and $\nabla w= 0$) when $w\neq 0$.
\end{proof}

\begin{corollary}\label{cor dictomy}
If $H>0$, and $p>1+\sqrt{\tfrac{4}{3}}$ (resp. $[|w|^{2p}]_W<\infty$) then $|w|^m\nabla\log H=\nabla |w|^{m}$ and $|w|^{2m-2}|\nabla w|^2=0$. Consequently, $\nabla\log H=\nabla \log |w|^m$ and $\nabla w=0$, or $w=0$ for $m=\tfrac{p-1}{2}$ ( resp. $m=p$).
\end{corollary}

\begin{proof}
    This follows from the above two propositions.
\end{proof}

\section {Proof of and Theorem \ref{main theo} and Theorem \ref{cor}}\label{the}
In this section, we prove Theorem \ref{main 0}, Theorem \ref{main theo} and Theorem \ref{cor}.

\begin{proof}[Proof of Theorem \ref{main theo}]
Since $u$ has positive speed, $w(0)=H(0)>0$ by Lemma \ref{positivity of H}. Hence $w\neq 0$ in a neighborhood $U$ of $0$ by continuity of $w$. Thus $\nabla\log\tfrac{|w|^m}{H}=0$ and $\nabla w=0$ in $U$ by Corollary \ref{cor dictomy}, i.e. $\tfrac{|w|^m}{H}=c_1>0$ and $|w|=c_2>0$ in $U$ for some positive constants $c_1,c_2$. By continuity, $w$ doesn't change sign in $U$. Thus, the set $B:=\{|w|=c_2\}$ is an nonempty open set. On the other hand, $B$ is a closed set by continuity. Thus $B=\mathbb{R}^n$. Plugging $c_2$ into the equation (\ref{eq1}), we get $c_2=\kappa=(\tfrac{1}{p-1})^{\frac{1}{p-1}}$. Moreover, since $w(0)>0$, $w=\kappa$.
\end{proof}

\begin{proof}[Proof of Theorem \ref{main 0}]
 This follows from (\ref{wat}), Theorem \ref{main theo} and a change of variable.
\end{proof}

\begin{proof}[Proof of Theorem \ref{cor}]
  Suppose $(a,T)$ is a blow up point of $u$. Let $w(y,s)=w_{a,T}(y,s)$, where $D_{a,T,\Omega}$ is defined in (\ref{wat}) and ($\ref{datomega}$). By Corollary 3.4 of \cite{fm}, $a$ is contained in a compact subset $K$ of $\Omega$. Fix an open subset $\Omega'$ of $\Omega$ such that $K\Subset\Omega'\Subset\Omega$. 
  By maximum principle and Theorem 4.2 of \cite{fm}, 
  $$0\leq u(x,t)\leq \tfrac{C(\varphi,n,p,\Omega)}{(T-t)^{\frac{1}{p-1}}}\text{ in }\Omega\times (0,T)$$ 
  for some universal constant $C(\varphi,n,p,\Omega)$ depending only on $\varphi,n,p,\Omega$. Equivalently, 
  $$0\leq w(y,s)\leq C(\varphi,n,p,\Omega)\text{ for }(y,s)\in D_{a,T,\Omega}.$$ 
  By Proposition 1' of \cite{gk85 asy sim}, 
  \begin{align*}
      |\nabla w|+|\nabla^2 w|\leq C'(\varphi,n,p,\Omega,\Omega'),\\
      |w_s|\leq C'(\varphi,n,p,\Omega,\Omega')(1+|y|),
  \end{align*}
  for $(y,s)\in D_{a,e^{-1}T,\Omega'}$. Applying Schauder theory for linear parabolic equations to (\ref{eq1}) yields
  \begin{align*}
      |\nabla^2 w|_{C^{2,\alpha}}\leq C''(\varphi,n,p,\Omega,\Omega',\alpha),\\
      |w_s|_{C^\alpha}\leq C''(\varphi,n,p,\Omega,\Omega',\alpha)(1+|y|),
  \end{align*}
  for $(y,s)\in D_{a,e^{-1}T,\Omega'}$. For any sequence $s_i\to\infty$, define $$w^{(i)}(y,s):=w(y,s+s_i), (y,s+s_i)\in D_{a,e^{-1}T,\Omega'}.$$
  By Arzel\`a-Ascoli theorem, there is a subsequence of $\{w^{(i)}\}_{i=1}^{\infty}$ (still denoted by $\{w^{(i)}\}_{i=1}^\infty$) which converges to a solution $\hat{w}$ of (\ref{eqw}) in $C^2_{loc}(\mathbb{R}^{n+1})$ as $i\to\infty$, with the estimate 
  \begin{align*}
     |\nabla^2 \hat{w}|_{C^{2,\alpha}}\leq C''(\varphi,n,p,\Omega,\Omega',\alpha),\\
     |\hat{w}_s|_{C^\alpha}\leq C''(\varphi,n,p,\Omega,\Omega',\alpha)(1+|y|)
  \end{align*}
  for $(y,s)\in \mathbb{R}^n\times(-\log T+1,\infty)$. On the other hand, recall the energy functional in \cite{gk85 asy sim},
\begin{align*}
    E(w(s)):=\int_{D_s}\big[\tfrac{1}{2}|\nabla w|^2+\tfrac{1}{2(p-1)}|w|^2-\tfrac{1}{p+1}| w|^{p+1}\big]\rho dy
    \end{align*}
    and
\begin{align*}
   E(\hat w(s)):=\int_{\mathbb{R}^n\times\{s\}}\big[\tfrac{1}{2}|\nabla \hat w|^2+\tfrac{1}{2(p-1)}|\hat w|^2-\tfrac{1}{p+1}| \hat w|^{p+1}\big]\rho dy,
\end{align*}
where $D_s:=D_{a,T,\Omega}\cap (\mathbb R^{n}\times \{s\})$, $\rho(y)=(4\pi)^{-\frac{n}{2}}e^{-\frac{|y|^2}{4}}$. Thanks to the exponential decay of $e^{-\frac{|y|^2}{4}}$, and $w^{(i)}(\cdot,s)\to \hat w(\cdot,s)$ in $C^{2}_{loc}(\mathbb R^n)$, we obtain 
  $$E(\hat{w}(s))=\lim_{i\to\infty}E(w^{(i)}(s)).$$
  Since $\Omega$ convex, it is star-shaped with respect to $a$, and $E(w(s))$ is monotone non-increasing in $s$ by (2.18) of \cite{gk85 sim var}. Thus, $E(\hat{w}(s))$ is independent of the sequence $\{s_i\}$, and 
  $$E(\hat{w}(s))=\lim_{i\to\infty}E(w^{(i)}(s))=\lim_{s\to\infty}E(w(s))$$ 
  is independent of $s$. Moreover, since $|\hat{w}(s)|_{C^1(\mathbb{R}^n)}\leq C''$, $(s>-\log T+1)$, every term in $E(\hat{w}(s))$ is finite. So, the monotonicity formula 
\begin{equation}\label{mono}
    \int_a^b\int_{\mathbb{R}^n}|\hat w_s(s,y)|^2\rho dyds=E(\hat w(a))-E(\hat w(b)),\quad\forall a,b\in\mathbb R,
\end{equation}
  holds following the same proof of proposition 3 of \cite{gk85 asy sim}. This implies that $\hat{w}_s\equiv 0$ on $\mathbb{R}^{n+1}$. That is, $\hat{w}$ is a classical solution of (\ref{eq1}) independent of $s$ on $\mathbb{R}^n$. 
  
 Secondly, $u_t(x,0)\geq0$ in $\Omega$ implies that $u_t(x,t)\geq 0$ for $(x,t)\in\Omega\times(0,T)$ by maximum principle. By the second equation of (\ref{ui ut}), $(\tfrac{1}{p-1}w+\tfrac{1}{2}y_iw_{i})(y,s)=(T-t)^{\frac{p}{p-1}}u(x,t)\geq 0$, $(y,s)\in D_{a,T,\Omega}$, which implies that $\hat{H}:=\tfrac{1}{p-1}\hat{w}(y)+\tfrac{1}{2}y_i\hat{w}_i(y)\geq 0$, $y\in\mathbb{R}^n$ by passing to the limit. Since $L_w\hat{H}=\hat{H}$, the Harncak inequality implies that $\hat{H}\equiv0$ or $\hat{H}>0$ in $\mathbb{R}^n$. If $\hat{H}\equiv 0$ in $\mathbb{R}^n$, we have $\Delta\hat{w}+p|\hat{w}|^{p-1}\hat{w}=\hat{H}=0$ by (\ref{eq1}). Since $\hat{w}\geq 0$, and $\frac{1}{p-1}\hat{w}(0)=\hat H(0)=0$, using Harnack inequality again, we get $\hat{w}\equiv 0$ in $\mathbb{R}^n$. If $\hat{H}>0$ in $\mathbb{R}^n$, we note that $\hat{w}\leq C$ by the previous paragraph. In particular, $\int_{\mathbb{R}^n}|\hat{w}|^{2p}e^{-\frac{|y|^2}{4}}dy<\infty$. Thus, we can apply Theorem \ref{main theo} to conclude that $\hat{w}\equiv \kappa$. 
 
At last, we note that $E(\hat{w})=\lim_{s\to\infty} E(w(s))$ is independent of $s_i$ and
 \begin{align*}
     E(\kappa)
     =&(\tfrac{1}{2}-\tfrac{1}{p+1})\kappa^{p+1}\int_{\mathbb{R}^n}e^{-\frac{|y|^2}{4}}dy>0=E(0),\quad p>1.
 \end{align*}
Thus,$\hat{w}$ is also independent of the sequence $\{s_i\}$. This implies that $w(y,s)\to 0$ or $\kappa$ as $s\to\infty$ in $C^{2,\alpha}_{loc}(\mathbb{R}^n)$. However, the first case can't happen by \cite{giga89}, since $(a,T)$ is a blowup point. The $C^{\infty}_{loc}$ convergence follows from a standard bootstrapping argument.  
\end{proof}

\section{Positive speed and linearly stability of self-similar solutions}
In this section, we define linear stability of self-similar solutions and discuss its relation with positive speed. As usual, we assume that $u$ is a smooth self-similar solution of \eqref{eq0} w.r.t. $(a,T)$ on $\mathbb R^n\times(0,T)$, and $L=L_w$ is the linearized operator of $F$ defined in \eqref{op Lw}.
\begin{definition}    
A smooth self-similar solution $u$ of (\ref{eq0}) is linearly stable if the only possible unstable eigenfunctions of $L$ corresponds to the re-centering of space and time\footnote{Here, a nonzero $C^2$ function $v$ is called an unstable eigenfunction of $L$ if $Lv=-\lambda v$ on $\mathbb R^n$ with $\lambda<0$. We allow the possibility that $L$ has no unstable eigenfunctions, that is, $L$ has no eigenfunctions with negative eigenvalue.}.
\end{definition}
By Lemma \ref{lemma eig t sp}, we know that $w_i$ ($i=1,2,\cdots n$) and $\tfrac{1}{p-1}w+\tfrac{1}{2}y\cdot\nabla w$ are the possible (when they are not identically zero) eigenfunctions of $L$ which correspond to the re-centering of space and time variable respectively. Thus, we have the equivalent definition of linearly stable self-similar solutions.
\begin{definition}
    Suppose $u$ is a smooth self-similar solution of (\ref{eq0}) on $\mathbb{R}^n\times (0,T)$ w.r.t. $(a,T)$, it is called linearly stable if and only if the only possible unstable eigenfunctions of $L$ are $\frac{1}{p-1}w+\frac{1}{2}y\cdot\nabla w$ and $w_i$ $(i=1,2\cdots,n)$, where $w=w_{a,T}$ is defined in (\ref{wat}).
\end{definition}
To analyze the eigenfunctions via calculus of variations, we need to introduce appropriate Hilbert spaces and restrict ourselves to more specific cases. Let $\langle\cdot,\cdot\rangle_W$ be the inner product defined in \eqref{eq we inn}. Define
\begin{equation}
\begin{aligned}
     \langle f,g\rangle_{W,1}=\langle f,g\rangle_W+\sum_{i=1}^n\langle \nabla_i f,\nabla_i g\rangle_W,\\
     \|f\|_{W,0}=\langle f,f\rangle_W^{\frac{1}{2}},\quad \|f\|_{W,1}=\langle f,f\rangle_{W,1}^{\frac{1}{2}}
\end{aligned}
\end{equation}
for $f,g\in C_c^\infty(\mathbb R^n)$, where . Let $H_W^0(\mathbb R^n)$, $H_W^1(\mathbb R^n)$ be the Hilbert space given by completing $C_c^\infty(\mathbb R^n)$ by using $\|\cdot\|_{W,0}$ and $\|\cdot\|_{W,1}$ respectively.
\begin{lemma}
    For all $v\in C_c^\infty(\mathbb R^{n})$,
    \begin{equation}
         \int_{\mathbb R^{n}} v^2|y|^2 e^{-\frac{|y|^2}{4}}dy\leq 16\int_{\mathbb R^n}|\nabla v|^2 e^{-\tfrac{|y|^2}{4}}dy+4n\int_{\mathbb R^{n}}v^2e^{-\frac{|y|^2}{4}}dy.
    \end{equation}   
\end{lemma}
\begin{proof}
    Since $v$ has compact support, we can choose $R$ large such that $v$ is supported in $B_{R}(O)$. By divergence theorem,
    \begin{equation}
    0=\int_{\mathbb R^n} \dv(yv^2e^{-\tfrac{|y|^2}{4}})dy=\int_{\mathbb R^n} (nv^2+2v\langle \nabla v,y\rangle-\tfrac{v^2}{2}|y|^2)e^{-\tfrac{|y|^2}{4}}dy 
    \end{equation}
    By rearranging terms and using Young's inequality,
    \begin{align*}
        \tfrac{1}{2}\int_{\mathbb R^n}v^2|y|^2 e^{-\tfrac{|y|^2}{4}}dy\leq n\int_{\mathbb{R}^n}v^2e^{-\tfrac{|y|^2}{4}}dy+4\int_{\mathbb R^n}|\nabla v|^2e^{-\tfrac{|y|^2}{4}}dy+\tfrac{1}{4}\int_{\mathbb R^n}v^2|y|^2e^{-\tfrac{|y|^2}{4}}dy
    \end{align*}
    Rearranging the above inequality gives the desired estimate.
\end{proof}
\begin{lemma}\label{lemma com emb}
    The natural embedding $\iota:H^1_W(\mathbb R^n)\hookrightarrow H^0_W(\mathbb R^n)$ is compact.
\end{lemma}
\begin{proof}
    The proof is similar to that of proposition B.2 in \cite{BW17} by using the above lemma. 
\end{proof}
Then we consider the min-max characterization of the first eigenvalue of $L$. To do so, we need to define the weak solution of 
\begin{equation}\label{eq weak so}
    Lv=f
\end{equation}
for $f\in H^0_W(\mathbb R^n)$ in $H^1_W(\mathbb{R}^n)$. For this purpose, we assume that there is a constant $C>0$ such that
\begin{equation}\label{eq bound}
    |u(x,t)|\leq\frac{C}{(T-t)^{\frac{1}{p-1}}} \text{ on }\mathbb R^n\times(0,T)
\end{equation}
Or equivalently,
\begin{equation}\label{eq bound 2}
    |w(y)|\leq C\text{ for }y\in\mathbb R^{n}.
\end{equation}
Then we define 
\begin{definition}
    $v\in H^1_W(\mathbb R^n)$ is said to be a weak solution of \eqref{eq weak so} if 
    \begin{equation}
        \int_{\mathbb R^n}(\nabla v\nabla\phi+\tfrac{1}{p-1}v\phi-p|w|^{p-1}v\phi)e^{-\frac{|y|^2}{4}}dy=-\int_{\mathbb R^n}f\phi e^{-\frac{|y|^2}{4}} dy
    \end{equation}
    for all $\phi\in C_c^\infty(\mathbb R^n)$.
\end{definition}
Since $|w|$ is assumed to be bounded, every term in the above equality is finite, and the weak solution is well defined in $H^1_W(\mathbb R^n)$. Moreover, Lemma \ref{lemma com emb} and the standard theory for compact self-adjoint operators imply that $L$ has discrete eigenvalues $\lambda_1<\lambda_2\leq \lambda_3\cdots\leq \lambda_m\cdots\to\infty$ with eigenfunctions $\{v_i\}_{i=1}^\infty$ which form a basis of $H^0_W(\mathbb R^n)$. Also, the first eigenvalue of $L$ is given by 
\begin{equation}
    \lambda_1=\inf_{v\in H^1_W(\mathbb R^n)\setminus\{0\}}\frac{\int_{\mathbb R^n}(|\nabla v|^2+\frac{1}{p-1} v^2-p|w|^{p-1}v^2)e^{-\frac{|y|^2}{4}}dy}{\int_{\mathbb R^n}  v^2 e^{-\frac{|y|^2}{4}}dy}.
\end{equation}
Next, we state a lemma and a theorem about the first eigenfunction and eigenvalue of $L$. Before proving them, we note that by \eqref{eq bound 2}, Proposition 1' of \cite{gk85 asy sim} implies that 
\begin{equation}\label{eq boun de}
    |\nabla w|+|\nabla^2 w|\leq C'\text{ on }\mathbb R^n
\end{equation}
for some constant $C'>0$. Then we can use this bound and the method in \cite{cm12} to get
\begin{lemma}\label{lemma char 1st eig fun}
    There is positive function $v$ on $\mathbb R^{n}$ with $Lv=-\lambda_1 v$. Furthermore, if $\hat v$ is in $H^1_W(\mathbb{R}^n)$ and $L\hat v=-\lambda_1 \hat v$, then $\hat v=Cv$ for some $C\in\mathbb R$.
\end{lemma}
\begin{proof}
    The proof is similar to that of Lemma 9.25 of \cite{cm12} if we replace $\frac{1}{2}$ by $-\frac{1}{p-1}$ and $|A|^2$ by $p|w|^{p-1}$ there.
\end{proof}
\begin{theorem}\label{theorem char bot of eigv}
    If $H:=\tfrac{1}{p-1}w+\frac{1}{2}y\cdot\nabla w$ changes sign, then $\lambda_1<-1$.
\end{theorem}
\begin{proof}
    The proof is similar to that of Theorem 9.36 of \cite{cm12} if we replace $\frac{1}{2}$ by $-\frac{1}{p-1}$ and $|A|^2$ by $p|w|^{p-1}$ there. In fact, by \eqref{eq boun de}, $|A|:=(p|w|^{p-1})^{\frac{1}{2}}$, $H$, $\nabla H$ are in the weighted $L^2$ space, and the proof of Theorem 9.36 in \cite{cm12} goes through. 
\end{proof}

\begin{corollary}\label{coro lin sta con}
     Suppose $u$ is a smooth self-similar solution of (\ref{eq0}) on $\mathbb{R}^n\times (0,T)$ w.r.t. $(a,T)$ satisfying \eqref{eq bound}, and $u$ is linearly stable. Then of $w\equiv 0$ or $\pm\kappa$.
\end{corollary}
\begin{proof}
     Since $u$ is linearly stable, the possible negative eigenvalues has eigenfunctions comes from the re-centering of time and space variable respectively, which are $H:=\frac{1}{p-1}w+\frac{1}{2}y\cdot\nabla w$ and $w_i(i=1,2\cdots,n)$ by Lemma \ref{lemma eig t sp}.
    By \eqref{eq bound 2} and \eqref{eq boun de}, $H,w_i\in H^1_W(\mathbb R^n)$ ($i=1,2,\cdots,n$). Thus $-1$ and $-\frac{1}{2}$ are the only two possible negative eigenvalues of $L$.  In particular, the first eigenvalue $\lambda_1\geq-1$. 
    If $H$ changes sign, then $\lambda_1<-1$ by theorem \ref{theorem char bot of eigv}, which is a contradiction. Thus, $H$ doesn't change the sign. 
    
    If $H\not\equiv 0$, then $H$ is the first eigenfunction, and $\lambda_1=-1$.  Since $H\in H^1_W(\mathbb R^n)$, the uniqueness in Lemma \ref{lemma char 1st eig fun} implies that $H>0$ (or $H<0$) on $\mathbb R^n$. Thus, $w=\pm \kappa$ by Theorem \ref{main theo}.
    
    Conversely, if $H\equiv 0$, let $(r,\theta)\in\mathbb R^+\times\mathbb S^{n-1}$ be the spherical coordinates on $\mathbb R^n$, then 
    \begin{align*}
    &\tfrac{1}{p-1}(wr^{\frac{2}{p-1}})+\tfrac{1}{2}y\cdot\nabla (wr^{\frac{2}{p-1}})\\
    =&\tfrac{1}{p-1}wr^{\frac{2}{p-1}}+\tfrac{1}{2}y\cdot(r^{\frac{2}{p-1}}\nabla w+\frac{2}{p-1}wr^{\frac{2}{p-1}-1}\frac{y}{r})\\
    =&\frac{1}{p-1}wr^{\frac{2}{p-1}}.
    \end{align*}
That is, 
\begin{align*}
    \frac{r}{2}\frac{\partial(wr^{\frac{2}{p-1}})}{\partial r}=\frac{1}{2}y\cdot\nabla (wr^{\frac{2}{p-1}})=0.
\end{align*}
This implies that $wr^{\frac{2}{p-1}}=f(\theta)+C_1$ for some constant $C_1$ and smooth function $f(\theta)$ defined on $\mathbb S^{n-1}$. Equivalently $w=r^{-\frac{2}{p-1}}(f(\theta)+C_1)$. Letting $r\to 0$ and using the fact that $w(0)=(p-1)H(0)=0$, we have $f(\theta)+C_1\equiv0$. Thus, $w=r^{-\frac{2}{p-1}}(f(\theta)+C_1)\equiv0$. 
\end{proof}
We have the following result which relates linearly stable self-similar solutions and self-similar solutions with positive speed.
\begin{corollary}
     Suppose $u$ is a smooth self-similar solution of (\ref{eq0}) on $\mathbb{R}^n\times (0,T)$ w.r.t. $(a,T)$ satisfying \eqref{eq bound}, and $u$ is linearly stable which is not identically zero, then either $-u$ or $u$ has positive speed. 
\end{corollary}
\begin{proof}
    If $u$ is a non-zero self-similar solution, then $w=(T-t)^{\frac{1}{p-1}}u$ is not identically zero. By the corollary above, $w=\kappa$ or $-\kappa$. By \eqref{ui ut} , $u_t=\tfrac{1}{p-1}(T-t)^{-\frac{p}{p-1}}\kappa$ or $-u_t=\tfrac{1}{p-1}(T-t)^{-\frac{p}{p-1}}\kappa$ for $(x,t)\in\mathbb R^n\times(0,T)$. That is, $u$ or $-u$ has positive speed.
\end{proof}


\begin{thebibliography}{99}
\bibitem{be}J. Bebernes, D. Eberly, {\em A description of self-similar blow-up for dimensions $n\geq 3$}, Ann. Inst. H. Poincar\'e Anal. Non Lin\'eaire(1988), no.1, 1-21.
\bibitem{BW17}J. Bernstein, L. Wang, {\em A topological property of asymptotically conical self-shrinkers of small entropy}. Duke Math. J.166(2017), no.3, 403-435.
\bibitem{bq}C.J. Budd, Y. Qi, {\em The existence of bounded solutions of a semilinear elliptic equation}, J. Differential Equations 82 (1989) 207-218. 
\bibitem{cm12} T. H. Colding, W. P. Minicozzi, II, {\em Generic mean curvature flow I: generic singularities}, Ann. of Math. (2) 175 (2012), no. 2, 755-833.
\bibitem{fp}M. Fila, A. Pulkkinen, {\em Backward self-similar solutions of supercritical parabolic equations}, Appl, Math. Lett. 22 (2009) 897-901. 
\bibitem{fm}A. Friedman, B. McLeod, {\em Blow-up of positive solutions of semilinear heat equations}, Indiana Univ. Math. J.34(1985), no.2, 425-447.
\bibitem{gp}V. A. Galaktionov, S. A. Posashkov, {\em The equation $u_t=u_{xx}+u^\beta$. Localization, asymptotic behavior of unbounded solutions}, Akad. Nauk SSSR Inst. Prikl. Mat. Preprint(1985), no.97, 30 pp.
\bibitem{gv}V.A. Galaktionov, J.L. Vázquez, {\em Continuation of blow-up solutions of nonlinear heat equations in several space dimensions}, Comm. Pure Appl. Math. 50 (1997) 1-67.
\bibitem{giga89}Y.Giga, {\em A local characterization of blowup points of semilinear heat equations}. Recent topics in nonlinear PDE, IV (Kyoto, 1988), 1-14. North-Holland Math. Stud., 160, North-Holland Publishing Co., Amsterdam, 1989.
\bibitem{gk85 asy sim}Y. Giga, R. V. Kohn, {\em Asymptotically self-similar blow-up of semilinear heat equations}, Comm. Pure Appl. Math.38(1985), no.3, 297-319.
\bibitem{gk85 sim var}Y. Giga, R. V. Kohn, {\em Characterizing blowup using similarity variables}, Indiana Univ. Math. J.36(1987), no. 1, 1-40.
\bibitem{gss04 R}Y. Giga, S. Matsui, S. Sasayama, Satoshi, {\em Blow up rate for semilinear heat equations with subcritical nonlinearity}, Indiana Univ. Math. J.53(2004), no.2, 483-514.
\bibitem{gss04 c} Y. Giga, S. Matsui, S. Sasayama, {\em On blow-up rate for sign-changing solutions in a convex domain}. Math. Methods Appl. Sci.27(2004), no.15, 1771-1782.
\bibitem{l}L.A. Lepin, {\em Self-similar solutions of a semilinear heat equation}, Mat. Model. 2 (1990) 63-74.
\bibitem{mi}N. Mizoguchi, {\em Nonexistence of backward self-similar blowup solutions to a supercritical semilinear heat equation}, J. Funct. Anal. 257 (2009) 2911-2937.
\bibitem{mi2}N. Mizoguchi, {\em On backward self-similar blow-up solutions to a supercritical semilinear heat equation}, Proc. Roy. Soc. Edinburgh Sect. A 140 (2010) 821-831. 
\bibitem{ns}Y. Naito, T. Senba, {\em Existence of peaking solutions for semilinear heat equations with blow-up profile above the singular steady state}, Nonlinear Anal. 181 (2019) 265-293. 
\bibitem{qui20}P. Pol\'a\v cik, P. Quittner, P. {\em On the multiplicity of self-similar solutions of the semilinear heat equation}, Nonlinear Anal.191(2020), 111639, 23 pp.
\bibitem{QS book}P. Quittner and P. Souplet, {\em Superlinear Parabolic Problems. Blow-up, Global Existence and Steady States}, Birkh\"auser Advanced Texts: Basler Lehrbücher, Birkh\"auser Verlag, Basel, 2007.
\bibitem{t}W.C. Troy, {\em The existence of bounded solutions of a semilinear heat equation}, SIAM J. Math. Anal. 18 (1987) 332-336.
\bibitem{www}K. Wang, J. Wei, K. Wu, {\em F-stability, entropy and energy gap for supercritical Fujita equation}, J. Reine Angew. Math. 822 (2025), 49–106.
\bibitem{www25}K. Wang, J. Wei, K.Wu, {\em A Liouville theorem for supercritical Fujita equation and its applications}, arXiv:2501.03574, to appear in Indiana Univ. Math. J.
\bibitem{WZZ}J. Wei, Q. Zhang, Y. Zhang, {\em Lecture notes on the parabolic gluing method}, online notes, 2023.

\end{thebibliography}
\end{document}